\documentclass{amsart}

\usepackage{amssymb, amsmath, latexsym, a4}
\usepackage[latin1]{inputenc}
\usepackage[all]{xy}

\usepackage{tikz}
\usepackage{pinlabel}
\usepackage{amsfonts}
\usepackage{amscd}
\usepackage{txfonts}
\usepackage{geometry}

\usetikzlibrary{matrix,calc,arrows}
\newcommand*\Z{\mathbb{Z}}

\newcommand{\athir}[2]{\displaystyle \prod_{C_{#2}\subset G} H^{#1}( C_{#2},M)}

\newtheorem{De}{Definition}[section]
\newtheorem{Th}[De]{Theorem}
\newtheorem{Pro}[De]{Proposition}
\newtheorem{Le}[De]{Lemma}
\newtheorem{Co}[De]{Corollary}
\newtheorem{Rem}[De]{Remark}

\numberwithin{equation}{subsection}

\def\ss{\sigma}
\def\xto#1{\xrightarrow[]{#1}}

\def\id{\sf Id}
\def\ee{\epsilon }

\newcommand{\w}{\mbox{\tiny $\wedge$ } }
\def\t{\otimes }

\def\oo{\omega }
\def\la{\lambda }
\def\tw{{ ^{\sf tw }}}
\def\co{{ ^{\sf c }}}

\def\1{^{-1}}

\def \hom{\mathop{\sf Hom}\nolimits}
\def \ext{\mathop{\sf Ext}\nolimits}

\def\al{\alpha}

\def\La{\Lambda }

\begin{document}

\title{Symmetric cohomology of groups and Poincar\'e duality}
\author{Mariam Pirashvili} \author{Teimuraz Pirashvili}
\thanks{The second author remembers fondly the lectures by A. Zarelua on commutative algebra in 1975, which were given in Georgian, which is not the native language of A. Zarelua.}
\address{Mariam Pirashvili,	laboratory LIP6 of Sorbonne University} 
\email{mariam.pirashvili@lip6.fr}

\address{Teimuraz Pirashvili, The University of Georgia,  Tbilisi, Georgia}
\email{t.pirashvili@ug.edu.ge}
\maketitle
\begin{abstract} Let $G$ be a finite group of order $n$  and let $M$ be a $G$-module. We construct groups  $H_*^\varkappa(G,M)$ for which
	$H_k^\varkappa (G,M\tw) \cong H^{n-k-1}_\la(G,M),$
	where $M\tw$ is a twisting of a $G$-module $M$ defined in Section \ref{spd} and $H^{*}_\la(G,M)$ is a  variation of the group cohomology introduced by Zarelua \cite{zarelua}, which in many cases is isomorphic to the symmetric cohomology of groups  \cite{staic}. The groups $H_*^\varkappa(G,M)$  come together with transformations from Tate cohomology $$\varkappa_k:\hat{H}^{-k}(G,M)\to H_k^\varkappa(G,M), \, k\geq 0.$$  We show that one has an isomorphism  $\hat{H}^{0}(G,M\tw)\cong H^{n-1}_\la(G,M)$ and  an  exact sequence
	$$\bigoplus _{C_2\subset G} \hat{H}^{-1} (C_2,M\tw)\to \hat{H}^{-1} (G,M\tw)\to H^{n-2}_\la(G,M)\to 0,$$
	where the direct sum is taken over all subgroups of $G$ of order two. Moreover, if $p$ is an odd prime such that the group has no elements of order $q$, for any prime $q$ such that  $q<p$, then one has an isomorphism 
	$:H^{n-i-1}_\la(G,M)\xto{\cong } \hat{H}^i (G,M\tw), \, 0\leq i< p-1.$

\end{abstract} 



\section{Introduction}
Let $G$ be a group and $M$ be a $G$-module. In \cite{zarelua},\cite{staic} Zarelua and Staic proposed modified group cohomologies denoted respectively by $H_\la^*(G,M)$ and 
$HS^n(G,M)$. According to \cite{jalg}, these groups together with the classical cohomology fit in the commutative diagram
$$\xymatrix{H_\la^n(G,M)\ar[rr]^{\gamma^n}\ar[dr]_{\beta^n}&& HS^n(G,M)\ar[dl]^{\alpha^n}\\
	&H^n(G,M).}$$
All maps are isomorphisms in dimensions $0$ and $1$. By \cite{jalg}, the homomorphism $\gamma^n: H_\la^n(G,M)\to HS^n(G,M)$ is a split monomorphism in general, and an isomorphism if $0\leq n\leq 4,$ or $M$ has no elements of order two. It follows that there is a natural decomposition 
\begin{equation}\label{hs=hl+hd}
HS^n(G,M)\cong H_\lambda^n(G,M)\oplus H^n_\delta(G,M).
\end{equation}
Based on the results in \cite{jalg} one can show that the map $\alpha^n$ vanishes on the mysterious summand $H^n_\delta(G,M)$ and hence all information in $HS^n$ 
relevant to the classical cohomology is already in Zarelua's theory $H_\lambda^n(G,M)$. By \cite{staic_h2}
and	\cite{jalg} the map $\beta^2$ (and hence $\gamma^2$) is a monomorphism. Moreover, if $p$ is an odd prime such that $G$ has no elements of order $q$, for $1<q<p$, then $\beta^k$ is an isomorphism for all $k<p$ and $\beta^p$ is a monomorphism. 

It follows from the very definition that if $G$ is a finite group of order $n$, then $H_\lambda^k(G,M)=0$ for all $k\geq n$. It was claimed  \cite[Theorem 3.2] {zarelua} that if $G$ is an oriented  group (see Section \ref{spd}) and $A$ is a trivial $G$-module, then for any $k>0$ one has the Poincar\'e duality: 
$$H_k^\la(G,A)\cong H_\la^{n-k-2}(G,A).$$
However this cannot be true. In fact, if we take $G$ to be Klein's Vier group $G=C_2\times C_2$ and $k=1$ then LHS equals to $H_1(G,A)=G\otimes A=(A/2A)^2$, while RHS $H^1(G,A)=Hom(G,A)=(_2A)^2$ and these groups are not isomorphic in general. Here $_2A=\{a\in A|2a=0\}$.

The aim of this work is to investigate Poincar\'e type duality in Zarelua's cohomology. Unlike \cite[Theorem 3.2] {zarelua} we do not restrict ourselves to oriented groups and trivial $G$-modules, but consider all finite groups and all modules over them. We construct groups  $H_*^\varkappa(G,M)$ for which we have the Poincar\'e duality (\ref{pd})
$$H_k^\varkappa (G,M\tw) \cong H^{n-k-1}_\la(G,M),$$
where $M\tw$ is a twisting of a $G$-module $M$ defined in Section \ref{spd}. (For oriented groups, $M\tw=M$ for any $G$-module $M$.) The groups $H_*^\varkappa(G,M)$ come together with transformations from Tate cohomology $$\varkappa_k:\hat{H}^{-k}(G,M)\to H_k^\varkappa(G,M), \, k\geq 0,$$  which happen to be an isomorphism if $k=0$ and hence we have $$ H^{n-1}_\la(G,M) \cong \hat{H}^0 (G,M\tw).$$
We will prove that the transformation $\varkappa_1$ is always an epimorphism and $Ker(\varkappa_1) $ is controlled by the elements of order two in $G$. As a consequence we obtain the following exact sequence
$$\bigoplus _{C_2\subset G} \hat{H}^{-1} (C_2,M\tw)\to \hat{H}^{-1} (G,M\tw)\to H^{n-2}_\la(G,M)\to 0,$$
where the direct sum is taken over all subgroups of $G$ of order two.

For $k\geq 2$ the transformation $\varkappa_k$ factors  trough Zarelua's homology groups
$$\hat{H}^{-k}(G,M)= H_{k-1} (G,M)\xto{\beta_{k-1}} H^\la_{k-1} (G,M)\to H_{k}^\varkappa(G,M)$$
and hence Zarelua's claim is true iff $H^\la_{k} (G,M)\to H_{k+1}^\varkappa(G,M)$ is an isomorphism. We will show that if $G$ is a group and $p\geq 5$ is a prime such that $G$ has no elements of order $q$ for all primes $q<p$, then Zarelua's claim is true for all $0<k\leq p-3$.


\section{Koszul complex and Poincar\'e dulaity for  sets}\label{pdsets} For a group $G$ the exterior algebra $\La^*(\Z[G])$ plays a prominent role in the definition of Zarelua's exterior (co)homology of groups \cite{zarelua}. Hence we recollect some well-known facts about exterior algebras.  

Recall that the exterior algebra $$\Lambda^*(A)=\bigoplus _{n}\Lambda^n(A)$$ of an abelian group $A$ is the quotient algebra of the tensor algebra $T^*(A)$ with respect to the two-sided ideal generated by the elements of the form $a\otimes a \in T^2(A)=A\otimes A$. 

There are several (co)chain complexes arising from this construction. All of them depend on some choices. Firstly, we choose an element $a\in A$. Since $a\wedge a=0$, we  have a cochain complex $(\La^*(A), \delta_a)$:
\begin{equation}\label{a_la}
0\to \Z\xto{\delta_a^0} A \xto{\delta_a^1} \La^2(A)\to \cdots \to \La^{n-1}(A) \xto{\delta_a^{n-1}}
\La^n(A)\to 0
\end{equation}
where $\delta_a^k: \La^{k}(A)\to \La^{k+1}(A)$ is given by
$$\delta^k_a(a_1\wedge \cdots \wedge a_k)=(-1)^ka\wedge a_1 \wedge \cdots \wedge a_k.$$

Secondly, we can choose a homomorphism of abelian groups $f:A\to \Z$. Then we have the Koszul complex \cite[$\S$ 9. p. 147]{ch10}
$(\Lambda^*(A), \partial_*)$,
where $\partial_{m-1}:\La ^{m}(A)\to \La^{m-1}(A)$, $m\geq 0$, is given by
$$\partial_{m-1}( a_1\wedge \cdots \wedge a_m)=\sum_{i=1}^m (-1)^{i+1}f(a_i) a_1 \wedge \cdots \wedge  \hat{a_i}\wedge\cdots \wedge a_m.$$
Here, as usual, the hat $\, \hat{}\,$ denotes a missing value.

Next, we will specify the above general constructions.  

Let us fix a set $S$ and take $A=\Z[S]$ to be a free abelian group generated by $S$. We take $f:\Z[S]\to \Z$ the homomorphism for which $f(s)=1$ for all $s\in S$. In this case the Koszul chain complex looks as follows:
$$               \cdots   \to   \La^3(\Z[S]) \xto{\partial_2}     \La^2(\Z[S]) \xto{\partial_1}   \Z[S]\xto{\partial_0}\Z \to 0
$$
and for any $s_1,\cdots,s_m\in S$ we have
$$\partial(s_1\wedge \cdots\wedge s_m)=\sum_{i=1}^m(-1)^{i+1}s_1\wedge \cdots\wedge \hat{s_i}\wedge\cdots\wedge s_m.$$
We call it the \emph{Koszul complex of  $S$} and it will be denoted by ${\sf Kos}_*(S)$. 
The following fact is well-known, see for example \cite[$\S 9$. Proposition 1. p.147]{ch10}. 

\begin{Le}\label{2122511} Let  $S$ be a nonempty set. Then for any element $c\in S$ one has
$$\delta_c^{k-1}\partial_{k-1}-\partial_k\delta_c^k=(-1)^{k-1}\id_{\La^k( \Z[S])}.$$
In particular $H_*({\sf Kos}_*(S))=0$.
\end{Le}

Assume now that $S$ is finite and denote by $n$ the cardinality of $S$.
It is well-known that $\La^k(\Z[S])$ is a free abelian group of rank $\binom{n}{k}$. In particular, $\La^k(\Z[S])=0$ for all $k\geq n+1$ and $\La^n(\Z[S])$ is an infinite cyclic group and hence has two generators. We call the set $S$ with chosen generator $\omega$ of $\La^n(\Z[S])$ an \emph{oriented} set.

Clearly any  order $S=\{t_1<t_2<\cdots <t_n\}$ on $S$
gives an orientation $\oo=t_1\wedge \cdots \wedge t_n$.
Moreover, if for a $k$-element  subset $I=\{1\leq i_1<\cdots<i_k\leq n\}$ we set 
$$t_I=t_{i_1}\wedge \cdots \wedge t_{i_k}\in  \La^k(\Z[S]),$$
then the collection of $t_I$, $|I|=k$ form  a basis of $\La^k(\Z[S])$ as an abelian group. Two orders on $S$ define the same orientation if the corresponding permutation is even.

The Koszul complex  in our situation is simply the following  exact sequence
\begin{equation}\label{par_la}
0\to \La^n(\Z[S])\xto{\partial_{n-1}} \La^{n-1}(\Z[S])\to \cdots \to \La^2(\Z[S])\xto{\partial_1} \Z[S]\xto{\partial_0} \Z\to 0.
\end{equation}
 
 Denote by ${\mathcal N}$ the sum $\sum_{s\in S}s\in \Z[S]$. We can look again to the cochain complex $(\La^*(\Z[S]), \delta_{\mathcal N})$. Recall that the homomorphism     $\delta_{\mathcal N}^k:\La^k(\Z[S])
 \to \La^{k+1}(\Z[S])$ is given by $$\delta_{\mathcal N}^k(s_1\wedge \cdots \wedge s_k)= (-1)^k{\mathcal N}\wedge   s_1\wedge \cdots \wedge s_k.$$
 In this way  we obtain the cochain complex:
 \begin{equation}\label{en_la}
 0\to \Z\xto{\delta_{\mathcal N}^0} \Z[S] \xto{\delta_{\mathcal N}^1} \La^2(\Z[S])\to \cdots \to \La^{n-1}(\Z[S]) \xto{\delta_{\mathcal N}^{n-1}}
 \La^n(\Z[S])\to  0.
 \end{equation}
 
 Our next goal is to prove that the cochain complexes (\ref{par_la}) and (\ref{en_la}) are isomorphic. The isomorphism is given by the Hodge star operator, which depends on the orientation on $S$. Recall that the Hodge star operator is an important operation on the differential forms on a Riemannian manifold, see for example \cite{gdf}. In the context of group cohomology  it already appeared in \cite{zarelua}.
 
 To  describe the Hodge star operator in our circumstances, we choose an order $S=\{t_1<t_2<\cdots <t_n\}$ compatible with a given orientation. Thus  $\oo=t_1\wedge \cdots \wedge t_n$
 and then we define the Hodge star operator
$\al_k:\La^k(\Z[G])\to \La^{n-k}(\Z[G])$  by the conditions
\begin{equation}\label{waloo}
t_I\wedge \al_k(t_I)=\oo.
\end{equation}
(Classically, the image of $\nu\in \La^k(\Z[G])$ under $\al_k$ is denoted by $\ast \nu$).
It is well-known  that this map depends only on the orientation on $S$.

 It follows from the definition that for any subset $I=\{i_1<\cdots < i_k\}$ one has
 $$\al_k(t_I)=(-1)^{\psi(I)}t_{I\co}$$
 where  $I\co$ is the complement of $I$ in $\{1,\cdots,n\}$ 
 and $ \psi (I)=(i_1-1)+\cdots +(i_k-k).$
 
 The following fact is probably well-known and we are not claiming any originality. In any case, it owes very much to \cite[Lemma 3.6]{zarelua}.

\begin{Pro} \label{hodgestar}  The collection of Hodge star operators $\al_k$ are compatible with differentials $\partial_k$ and $\delta_{\mathcal N}^k$ in the sense that one has an isomorphism of (co)chain complexes of abelian groups:
$$\xymatrix{ 0 \ar[r] & \La^n(\Z[S])\ar[r]^{\partial_{n-1}}& \La^{n-1}(\Z[S])\ar[r] & \cdots \ar[r]  &\Z[S]\ar[r]^{\partial_0}& \Z\ar[r] &0 
	\\
0\ar[r] & \Z\ar[r]^{\delta_{\mathcal N}^0} \ar[u]^{\al_0}  & \Z[S] \ar[r]  \ar[u]^{\al_1} &\cdots \ar[r]& \La^{n-1}(\Z[S]) \ar[r]^{\delta_{\mathcal N}^{n-1}} \ar[u]^{\al_{n-1}} &
\La^n(\Z[S])\ar[r]\ar[u]^{\al_n}&  0.}
$$
\end{Pro}

\begin{proof} Take a subset $I=\{i_1<\cdots < i_k\}$, $k<n$. We need to show that $$\partial_{n-k-1}\al_k(t_I)= \al_{k+1} \delta_{\mathcal N}^k(t_I).$$
Since $\al_k$ depends only on the orientation of $S$, we can change the ordering on $S$, if it is necessary, to assume $I=\{1<\cdots<k\}$. In this case we obtain
\begin{align*} \al_{k+1} \delta_{\mathcal N}^k(t_I)&= \al_{k+1}\left ((-1)^k\sum_{i=1}^{n-k} t_{k+i}\wedge t_1\wedge \cdots \wedge t_k\right)\\
& =\sum_{k=1}^{n-k}  \al_{k+1}\left ( t_1\wedge \cdots \wedge t_k \wedge t_{k+i}\right )\\
&= \sum_{i=1}^{n-k}(-1)^{i-1}t_{k+1}\wedge \cdots \wedge \hat{t}_{k+i} \wedge \cdots \wedge t_n.
\end{align*}
On the other hand, we also have
\begin{align*}
\partial_{n-k-1}\al_k(t_I)&= \partial_{n-k-1}(t_{k+1}\wedge \cdots \wedge t_n)\\
&=\sum_{i=1}^{n-k}(-1)^{i-1}t_{k+1}\wedge \cdots \wedge \hat{t}_{k+i} \wedge \cdots \wedge t_n
\end{align*}
and we are done.
\end{proof}
\section{Classical and Zarelua's (co)homology of groups}

In this section we recall the definition of the classical and  Zarelua's (co)homology of groups.  

Let $G$ be a group. We consider $\Z$ as a $G$-module, with trivial action of $G$. Denote $P_k= \Z[G^{k+1}]$, which is considered as a $G$-module via the action
$$g(g_0,\cdots, g_k)=(gg_0,\cdots, gg_k)
$$
and consider the standard projective resolution of $\Z$ by $G$-modules  \cite{aw}
\begin{equation}\label{stpr}
\cdots \to P_k \xrightarrow{\partial}P_{k-1} \to\cdots \to P_0\xrightarrow{\epsilon}\Z,
\end{equation}
where 
%
the boundary map is given by
$$\partial(g_0,\cdots ,g_k)=\sum_{i=0}^k(-1)^i(g_0,\cdots ,g_{i-1},g_{i+1},\cdots , g_k),$$
and the mapping $\epsilon$ sends each generator $g$ to $1\in\Z$. 
For any $G$-module $M$ we apply the functor $\hom_G(-,M)$ (resp. $(-)\otimes_GM$ ) to the standard projective resolution to obtain the cochain complex $K^*(G,M)$ (resp. chain complex $K_*(G,M)$) and then one defines the  cohomology and homology of a group $G$ with coefficients in $M$ by
$$H^*(G,M)=H^*(K^*(G,M)) \quad {\rm and} \quad H_*(G,M)=H_*(K_*(G,M)).$$

Zarelua  uses the Koszul complex ${\sf Kos}^*(G)$ of Section \ref{pdsets} to define his (co)homology groups as follows.  Firstly, the abelian groups $\La^k(\Z[G])$, $k\geq 0$, have the natural $G$-module structure given by
$$g \cdot x_1\wedge \cdots \wedge x_k=
gx_1\wedge \cdots \wedge gx_k.$$
Let us denote by ${\sf Kos}^+_*(G)$ the chain complex
$$ \cdots   \to   \La^3(\Z[S]) \xto{\partial}     \La^2(\Z[S]) \xto{\partial}   \Z[S].
$$
By Lemma \ref{2122511} one can look at ${\sf Kos}^+_*(G)$ as a (non-projective) resolution of $\Z$ in the category of $G$-modules and then one puts
$$K^*_\la (G,M)=\hom_G({\sf Kos}^+_*(G), M), \quad K_*^\la (G,M) ={\sf Kos}^+_*(G)\otimes _G M.$$
Zarelua's (co)homology groups are defined by
$$H^*_\la(G,M):=H^*(K^*_\la (G,M))
\quad  {\rm and}\quad  H_*^\la(G,M):=H_*(K_*^\la(G,M)).$$
We have  a morphism of projective resolutions,  see \cite[Lemma 3.3]{jalg},
$$\xymatrix{\cdots \ar[r]& \Z[G^i] \ar[r]\ar[d]& \Z[G^{i-1}]\ar[r]\ar[d]&\cdots\ar[r] &  \Z[G^{2}] \ar[r]\ar[d]& \Z[G]\ar[d]^{Id}\\
\cdots \ar[r]& \Lambda^i\Z[G] \ar[r]& \Lambda^{i-1}\Z[G]\ar[r]&\cdots\ar[r] & \Lambda^2 \Z[G] \ar[r]& \Z[G]}$$
and hence one obtains natural transformations
$$\beta^*: H^*_\la(G,M)\to H^*(G,M)\quad {\rm and} \quad 
\beta_*:H_*(G,M)\to H_*^\la(G,M).$$
We refer to \cite{jalg} for extensive information about $\beta^*$. It was not stated in \cite{jalg} but all results obtained in \cite{jalg} have obvious analogues for $\beta_*$. In particular, $\beta_0,\beta_1$ are isomorphisms, while $\beta_2$ fits in the exact sequence
$$\bigoplus_{C_2\subset G} H_2(C_2,M)\to H_2(G,M)\xto{\beta_2} H^\la_2(G,M)\to 0,$$
where $C_2$ runs trough all subgroups of order two, etc. 

{\bf Example}. Based on the  exact sequence in \cite[Corollary 4.4i) ]{jalg} and \cite[Exercise 5. p. 66]{fg} we see that 
$H^2_\la (D_n,\mathbb{F}_2)=0$ for all dihedral groups $D_n$. 
\section{Tate cohomology and $H_*^\varkappa(G,M)$ groups}
In what follows we restrict ourselves to the case when $G$ is a finite group of order $n$. Recall that in this case we also have the Tate cohomology groups $\hat{H}^*(G,M)$, which combine the aforementioned homology and cohomology of groups. They are defined by
$$\hat{H}^k(G,M)=\begin{cases} H^k(G,M) & {\rm if } \ k>0\\
Coker \,{\sf N} & {\rm if } \ k=0\\
Ker\, {\sf N}, & {\rm if } \ k=-1\\
H_{-k-1}(G,M) & {\rm if } \ k<-1\\
\end{cases}
$$
where ${\sf N}$ is the norm homomorphism
$${\sf N}:H_0(G,M)\to H^0(G,M)$$
induced by the map $M\to M$ given by $m\mapsto \sum_{g\in G}gm={\mathcal N}m$. 

We also need the following equivalent description of Tate cohomology \cite{aw}. For any abelian group $A$ denote by $A^\vee$ the ''dual'' abelian group $\hom(A,\Z)$. It is well-known that the natural map $A\to (A^\vee)^\vee$ is an isomorphism if $A$ is a finitely generated free abelian group. Observe that if $A$ is a $G$-module, then $A^\vee$ is also a $G$-module with the action
$$(g\cdot \psi)(a)=\psi(g\1a)$$
where $a\in A$ and $\psi\in A^\vee$.

Since the standard resolution $P_*\to \Z\to 0$ splits as a chain complex of abelain groups, after dualizing we still obtain an exact sequence
$$0\to \Z \to P_*^\vee.$$
 We can eliminate $\Z$ to get an exact sequence
 $$\cdots \to P_1\to P_0\xto{\sf N} P_0^\vee \to P_1^\vee \to \cdots$$
and after reindexing $P_k:=P_{-k-1}^\vee$ for $k<0$, one obtains \cite[p.103]{aw} 
 $$\hat{H}^k(G,M)=H^k(\hom_G(P_*,M)), \,k\in\ Z.$$
Equivalently, for any $k\leq 0$ for the Tate cohomology groups $\hat{H}^k(G,M)$ we have  
\begin{equation}\label{tta}
\hat{H}^k(G,M)\cong H_k\left ( \hom_G(\Z,M)\leftarrow \hom_G(P_0^\vee, M) \leftarrow \hom_G(P_1^\vee, M) \leftarrow \cdots
\right ).
\end{equation}

Now we  introduce  groups $H_*^\varkappa(G,M)$ which are defined as follows.
We consider the exact sequence (\ref{en_la}) for $S=G$:
$$0\to \Z\xto{\delta_{\mathcal N}^0} \Z[G] \xto{\delta_{\mathcal N}^1} \La^2(\Z[G])\to \cdots \to \La^{n-1}(\Z[G]) \xto{\delta_{\mathcal N}^{n-1}}
\La^n(\Z[G])\to  0$$
which is denoted by $C^*(G)$. Thus $C^k=\La^k\Z[G]$, $0\leq k\leq n$.
\begin{Le} The boundary maps $\delta_{\mathcal N}^i$ are $G$-homomorphisms, $0\leq i\leq n-1$.
\end{Le}
\begin{proof} It is required to check
$${\mathcal N}(gx_1\wedge \cdots \wedge gx_i) =g({\mathcal N}(x_1\wedge \cdots \wedge x_i)).$$
Since $\sum_{h\in G}h=\sum_{h\in G}gh$, we have

\begin{align*} {\mathcal N}(gx_1\wedge \cdots \wedge gx_i)& = 
\sum_{h\in G}h\wedge gx_1\wedge \cdots \wedge gx_i\\
&=\sum_{h\in G}gh\wedge gx_1\wedge \cdots \wedge gx_i\\
&=g{\mathcal N}(x_1\wedge \cdots \wedge x_i).
\end{align*}
\end{proof}
Since $C^*(G)$ is a cochain complex of $G$-modules we can apply the functor $\hom_G(-,M)$ to obtain a new chain complex
$$\hom_G(\Z,M)\leftarrow \hom_G(\Z[G],M)\leftarrow  \hom_G(\La^2\Z[G],M)\leftarrow \cdots \leftarrow  \hom_G(\La^n\Z[G],M)\leftarrow 0,$$
which is denoted by $K_*^\varkappa(G,M)$. Now we set
$$H_*^\varkappa(G,M):=H_*(K_*^\varkappa(G,M)).$$

We are now going to construct natural transformations from the Tate cohomology and Zarelua's homology into $H^\varkappa_*(G,M)$. The construction is based on the following well-known lemma. 
\begin{Le} \label{4212}  Let $M$ be a $G$-module and $L_*$ a chain complex of finitely generated $G$-modules. Assume each $L_k$ is free as an abelian group. Then the map
$${\mathfrak t}: L_*\otimes _G M\to \hom_G(L_*^\vee, M)
$$
is a natural chain map, which is an isomorphism if each $L_k$ is a projective $G$-module, where $\mathfrak t$ is the composite of the following maps
$$L_*\otimes _GM=H_0(G, L_*\otimes M)\xto{\mathcal N}
H^0(G, L_*\otimes M)\xto{\mathfrak s} H^0(G, \hom(L_*^\vee,M))=\hom_G(L_*^\vee,M)$$	and ${\mathfrak s}:L_*\otimes M\to \hom(L_*^\vee,M)$ is given by
$${\mathfrak s}(x\otimes m) (\psi)	 =\psi(x)m.$$
Here $x\in L, m\in M$ and $\psi\in L^\vee=Hom(L,\Z)$.
\end{Le} 
\begin{proof} See  \cite[p.103]{aw} for an explicitly written proof of this fact, which implicitly is there. 
\end{proof}

By Lemma \ref{2122511} and Proposition \ref{hodgestar} we see that the exact sequence $C^*(G)$ splits as a chain complex of abelian groups. Hence we can dualize it to get an exact sequence 
$$0\leftarrow \Z \leftarrow C_1^\vee \leftarrow \cdots \leftarrow C_n^\vee \leftarrow 0.$$
In fact, this is an exact sequence of $G$-modules and hence it can be considered as a (nonprojective) resolution of $\Z$. By a well-known property of projective resolutions, the identity map $\Z\to \Z$ has a lifting as a morphism of chain complexes 
$$\xymatrix{
	0 &  \Z \ar[l]\ar[d]_{id}& P_0 \ar[l]\ar[d] &P_1 \ar[d]\ar[l]&\ar[l] &\cdots \\
	0 &  \Z \ar[l]& C_1^\vee \ar[l] &C_2^\vee \ar[l]&\ar[l] &\cdots 
}$$
By dualizing, we obtain the diagram 
$$\xymatrix{
	0 \ar[r] &  \Z \ar[r]\ar[d]_{id}& C_1  \ar[r]\ar[d] &C_2 \ar[d]\ar[r]& \cdots \\
	0 \ar[r] &  \Z \ar[r]& P_0^\vee \ar[r] &P_1^\vee \ar[r]& \cdots 
}$$
By applying the functor $\hom_G(-,M)$ and taking the homology functor one obtains natural homomorphisms $$\varkappa_k:\hat{H}^{-k}(G,M)\to H_k^\varkappa (G,M), \, k\geq 0.$$
The next aim is to prove that the maps $\varkappa_*$ factor through Zarelua's homology groups.
 To this end, we use Lemma \ref{4212} first when  $L_*=C_*^\vee$ and then when $$L= P^{aug}_*=\xymatrix{ 0 &  \Z \ar[l]& P_0 \ar[l] &P_1 \ar[l]&\ar[l] & \cdots}$$ to obtain the commutative diagram of chain complexes
 $$\xymatrix{P^{aug}_*\otimes_G M\ar[r] \ar[d] & \hom_G(P^{{aug}_*\vee},M)\ar[d]\\
 C_*^\vee\otimes_G M\ar[r]& \hom_G(C_*,M)
}$$
Since $P_n$ are projective $G$-modules, we see that the top map after taking homology induces an isomorphism in  dimensions $\geq 2$. So we obtain the diagram
$$\xymatrix{H_k(G,M)\ar[r]^{\cong}\ar[d]_{\beta_k} & \hat{H}^{-k-1}(G,M)\ar[d]_{\varkappa^{k+1}}\\
H^{\la}_k(G,M) \ar[r]^{\zeta_k} &H^{\varkappa}_{k+1}(G,M)
}$$
Hence $\varkappa_{k+1}=\zeta_{k}\beta_{k}$,{\tiny } $k\geq 1$. 
\begin{Le} \label{ttvkap}
	\begin{enumerate}
		\item [i)]  For any group $G$ and any $G$-module $M$ the homomorphism $$ \varkappa_0:\hat{H}^{0}(G,M)\to H_0^\varkappa (G,M)$$ is an isomorphism.
		
	\item [ii)] 
	Let $k\geq 2$ and let $G$ be a group such that the $G$-modules 
	$\La^i(\Z[G])$ are projective $G$-modules for all $2\leq i\leq k$. Then for any $G$-module $M$ the homomorphism $\varkappa_i$ is an isomorphism for all $0\leq i <k$.
		\end{enumerate} 
 \end{Le}

\begin{proof} i) By definition we have  an exact sequence
	$$M\xto{\sf N} H^0(G,M)\to \hat{H}^0(G,M)\to 0.$$
Next, we have  $$K_0^\varkappa(G,M)=\hom_G(\Z,M)=H^0(G,M)$$
and  $$K_1^\varkappa(G,M)=\hom_G(\Z[G],M)=M.$$
Since the boundary map  $K_1^\varkappa(G,M)\to K_0^\varkappa(G,M)$  is induced by the norm map, we obtain
	$$
	H_0^\varkappa(G,M)=\hat{H}^0(G,M).
	$$

ii) It follows from the basic properties of projective resolutions that the 
 augmented chain complexes $P_*\to \Z$ and $C_*^\vee$ are homotopy equivalent up to dimension $k$ and hence the result. 
\end{proof}

\begin{Co}\label{43}  If the group has no elements of order two, then $$\varkappa_1:\hat{H}^{-1}(G,M)\to H_1^\varkappa (G,M)$$ is an isomorphism. More generally, if $p$ is an odd prime, such that for all prime $q<p$ the group $G$ has no elements of order $q$, then $\varkappa_k$ is an isomorphism for all $k<p-1$.	  
\end{Co}
\begin{proof} In this case all $G$-modules $\La^i(\Z[G])$, $2\leq i\leq p-1$, are free, see the proof \cite[Theorem 4.2] {jalg}.   
\end{proof}

\section{A Poincar\'e duality}\label{spd}

In this section we will still  assume that $G$ is a finite group of order $n$.
We would like  to  apply  the results of Section \ref{pdsets} to group cohomology. Observe that in general the Hodge star operator $\al_k$ is not compatible with the action of $G$. To avoid this difficulty  we will twist the action.

For an element $g\in G$ denote by $\ell_g:G\to G$ the map given by $\ell_g(h)=gh$, $g,h\in G$. Since $\ell_g$ is a bijection, we can consider the sign of this permutation, which is denoted by $\ee(g)$. In this way we obtain a group homomorphism
$$\ee:G\to\{\pm 1\}$$

Following \cite{zarelua}, $G$ is called \emph{oriented} if $\ee$ is the trivial homomorphism. It is clear that any group of odd order or any perfect group is oriented, because there are no nontrivial homomorphisms into $\{\pm 1\}$. 

 Next, we choose a generator $\oo$ of $\La^n(\Z[G])$. We can take $\oo=1\wedge x_2\wedge \cdots \wedge x_n$, where $<$ is a chosen order on $G$:
$$G=\{1=x_1<x_2<\cdots <x_n\}.$$ 
Then for any $g\in G$ we have 
\begin{equation}\label{osign}
g\oo =\ee(g)\oo.
\end{equation}

Comparing the definitions, we see that $G$ is oriented iff the action of $G$ on $\La^n(\Z[G])$  is trivial, in other words if for any $1\leq i\leq n$ one has
$$1 \wedge x_2\wedge \cdots \wedge x_n=
x_i \wedge x_ix_2\wedge \cdots \wedge x_ix_n.$$

Now we introduce the twisting operation on $G$-modules.

If $A$ is a $G$-module, then $A\tw$ is the $G$-module which is $A$ as an abelian group, while the new action is given by
$$g\ast a=\ee(g)(ga).$$
For oriented groups $A\tw=A$. In general, we have the following isomorphism of $G$-modules
$$A\tw \cong A\otimes \Z\tw$$
and the map $\alpha_0$ from Section \ref{pdsets} induces the isomorphism of $G$-modules
 \begin{equation}\label{lanee}
 \Z\tw\xto{\al_0}\La^n(\Z[G]), \quad 1\mapsto \oo.
\end{equation}
Our next goal is to extend the isomorphism (\ref{lanee}) to other exterior powers. We claim that for any $x\in G$ one has
$$\al_k(gx_I) =\ee(g)g\al_k(x_I).$$
In fact, we have
$$(gx_I)\wedge \al_k(gx_I)=\oo=\ee(g)(\ee(g)\oo)=\ee(g)
(gx_I\wedge g\al_k(x_I))$$
and the claim follows. Here we used the identity (\ref{osign}). Thus, we have shown the following.

\begin{Le}\label{312511}
	For all $k\ge 0$ the map $\al_k$ induces an  isomorphism of $G$-modules
	$$\al_k:(\La^k(\Z[G]))\tw \to \La^{n-k}(\Z[G])$$
and hence we have an isomorphism of cochain complexes of $G$-modules:
$$\xymatrix{ 0 \ar[r] & \La^n(\Z[G])\ar[r]^{\partial_{n-1}}& \La^{n-1}(\Z[G])\ar[r] & \cdots \ar[r]  &\Z[G]\ar[r]^{\partial_0}& \Z\ar[r] &0 
		\\
0\ar[r] & \Z\tw\ar[r]^{\delta_{\mathcal N}^0} \ar[u]^{\al_0}  & \Z[G]\tw \ar[r]  \ar[u]^{\al_1} &\cdots \ar[r]& \La^{n-1}(\Z[G])\tw \ar[r]^{\delta_{\mathcal N}^{n-1}} \ar[u]^{\al_{n-1}} &
\La^n(\Z[G])\tw\ar[r]\ar[u]^{\al_n}&  0.}
$$
\end{Le}

After  applying the functor  $\hom_G(-,M)$ we see that  the cochain complex $K^*_\la(G,M)$ is isomorphic to the complex 
$$0\leftarrow \hom_G(\Z\tw,M)\leftarrow \hom_G(\Z[G]\tw,M)\leftarrow \cdots \leftarrow \hom_G(\La^{n-1}\Z[G]\tw,M).$$
Since the assignment $A\mapsto A\tw$ is an autoequivalence of the category of $G$-modules, we have $\hom_G(X\tw,Y)=\hom_G(X,Y\tw)$. Hence   $K^*_\la(G,M)$ is isomorphic to truncated $K_*^\varkappa(G,M)$, which is obtained by  deleting  the last group $\hom_G(\La^n\Z[G],M\tw)$. Thus we obtain the following result.

\begin{Pro}\label{pd} For any finite group $G$ and any $G$-module $M$ one has a Poincar\'e type duality:
$$\Delta_k:H^{n-k-1}_\la(G,M)\xto{\cong } H_k^\varkappa (G,M\tw),$$
for all  $0\leq k\leq  n-2$.
\end{Pro}

In particular, by Lemma \ref{ttvkap} and  Corollary \ref{43} we have the following facts.

\begin{Co} Let  $G$ be a finite group of order $n$ and $M$ be a $G$-module. 
	
	\begin{enumerate} 
		\item [i)] We have an isomorphism
		$$\Delta_0:H^{n-1}_\la(G,M)\xto{\cong } \hat{H}^0 (G,M\tw).$$
		\item [ii)] If $p$ is a prime such that the group has no elements of order $q$, for any prime $q$ such that $q<p$, then one has an isomorphism 
		$$\Delta_i:H^{n-i-1}_\la(G,M)\xto{\cong } \hat{H}^i (G,M\tw), \, 0\leq i< p-1.$$
	\end{enumerate}
\end{Co}


\section{Decompositions of $\La^2(\Z[G])$ and applications}
In this section we decompose the $G$-module $\La^2(\Z[G])$ as a direct sum of  submodules.  Denote by $\bar{G}$ the set $G\setminus\{1\}$. Clearly the elements of the form $1\wedge x$, $x\in \bar{G}$, generate $\La^2(\Z[G])$ as a $G$-module, because $a\wedge b=a(1\wedge c)$, where $c=a\1b$. Some of them will be $G$-free generators, meaning that the corresponding cyclic submodule is free, but not all of them, if $G$ has $2$-torsion.  In fact, if $o(x)=2$, then
$$1\wedge x+x(1\wedge x)=0.$$
Here $o(x)$ denotes the order of $x\in G$. Observe also that the relation $1\wedge x+x(1\wedge x\1)=0$ shows that  $1\wedge x$ and $1\wedge x\1$ define the same submodules.  

This observation suggests that we introduce further notations. The map $x\mapsto x\1$ defines an action of the cyclic group of order two on the set $\bar{G}$. Fixed elements under this action are exactly such $x\in G$ for which $o(x)=2$. The complement  subset $\{x\in G| x^2\not =1\}$ is a  $C_2$-subset of $\bar{G}$ on which the action is free.  We denote by ${\mathfrak O}_2(G)$
the corresponding set of orbits. For any orbit  $\xi$   we choose a representative element in $\xi$, which is denoted by $\ss(\xi)$. Thus  $\xi=\{\ss(\xi), \ss(\xi)\1 \}$.

For any $x\in \bar{G}$, define a $G$-module homomorphism $f_x:\Z[G]\to \La^2(\Z[G])$ by 
$$f_x(g)=g\wedge gx, \, g\in G.$$
Observe that for $o(x)=2$, we have 
$$f_x(1+x)=1\wedge x+x\wedge 1=0.$$
Thus, for such $x$ the map $f_x$ factors through
$$\hat{f}_x:\Z[G]/\Z[G](x+1)\to \Lambda^2(\Z[G]).$$
Now we set
$$W_1=\bigoplus_{o(x)=2}\Z[G]/\Z[G](x+1) $$
and
$$W_2=\bigoplus_{\xi\in {\mathfrak O}_2(G)} \Z[G],$$
where $x$ (resp. $\xi$) is running through the set of elements of order two (resp. ${\mathfrak O}_2(G)$).

It would be convenient to write a general element of $W_1$ (resp. $W_2$) as a sum $\sum_{o(x)=2}j_x(u)$ (resp. $\sum_{\xi}i_\xi(v)$ ). Here $u\in \Z[G]/\Z[G](x+1)$, $v\in \Z[G]$ and $j_x:\Z[G]/\Z[G](x+1)\to W_1$ and $i_\xi:\Z[G]\to W_2$ are standard inclusions into the direct sum.

Define the $G$-module homomorphisms
$$f_1:W_1\to \Lambda^2(\Z[G]) \quad {\rm and } \quad f_2:W_2\to \Lambda^2(\Z[G])$$
by
$$f_1\left (\sum_{o(x)=2}j_x(u)\right )=\sum_x \hat{f}_{x}(u) \quad 
{\rm and} \quad f_2\left (\sum_{\xi}i_\xi(v)\right )=\sum_\xi f_{\ss(\xi)}(v).$$
The maps $f_1$ and $f_2$ define the $G$-module homomorphism
$$f:W_{1} \oplus W_2 \to \Lambda^2(\Z[G]),$$
which is $f_1$ on $W_1$ and $f_2$ on $W_2$.
\begin{Le} \label{6112}
	The map $f:W_{1} \oplus W_2 \to \Lambda^2(\Z[G])$ is an isomorphism of $G$-modules.
\end{Le}
\begin{proof} By construction, $f$ is a $G$-module homomorphism. Hence we need to show that $f$ is an isomorphism of free abelian groups. To this end, also decompose the RHS as a direct sum of two summands $U\oplus V$, where $U$ (resp. V) is spanned as an abelian group by $a\wedge b$, $a,b\in G$ such that $o(a\1b)\not =2$ (resp.  $o(a\1b) =2$). We have  $f(W_1)\subset U$ and  $f(W_2)\subset V$. 
	
	Take $\xi\in {\mathfrak O}_2(G)$. Denote by $\tau_\xi$ the inclusion of $\Z[G]$ in $W_2$ corresponding to the summand indexed by $\xi$. The collection $\tau_\xi(g)$ form a basis of $W_2$, where $g\in G$. By definition $f$ sends this element to $g\wedge g\ss(\xi)$, which up to sign is a free generator of $V$ (because $o(g\1g\ss(\xi))\not =2$). Conversely, take any basis element $a\wedge b$ of $V$, where $a,b\in G$ and $a<b$ in a chosen order of $G$. We have $a\wedge b=a(1\wedge c)$, where $c=a\1b$. By definition of $V$ we have  $o(c)\not =2$. Denote by $\xi$ the class of $c$ in ${\mathfrak O}_2(G)$. Then $c=\sigma(\xi)$ or $c\1 =\sigma(\xi)$. In the first case $f_{\sigma(\xi)}(a)=a\wedge b$ and in the second case  $f_{\sigma(\xi)}(b)=-a\wedge b$. This shows that the matrix corresponding to $f$ in the chosen bases is a diagonal matrix with $\pm 1$ on the diagonal. Hence  $f$ induces an isomorphism $W_2\to V$.

	Let $x\in G$ be an element of order two. We set $C_{2}x=\{1,x\}$, the cyclic subgroup generated by $x$.
	First, we choose a section $\tau_x:G/C_2(x)\to G$ of the canonical map $G\to G/C_2(x)$. Then the  collection $\tau_x(\eta)$ form a basis of the abelian group $\Z[G]/\Z[G](x+1)$, where $\eta$ runs trough the set $G/C_2(x)$. Hence, the elements $\tau_{x}(\eta)$ form a basis of $W_1$, where $o(x)=2$.
	By definition the map $f$ sends a basis element $\tau_{x}(\eta)$  to $\tau(\eta)\wedge \tau(\eta)x$, which is a basis element (up to sign) of $U$. Conversely, take any basis element $a\wedge b$ of $U$, where $a,b\in G$ and $a<b$ in a chosen order of $G$. We have $a\wedge b=a(1\wedge c)$, where $c=a\1b$ and  $o(c)=2$.   Denote by $\bar{a}$ the class of $a$ in $G/C_2(c)$. Then $\tau_c(\bar{a})$ is either $a$ or $ac$ and we have 
	$f_c(\tau_c(\bar{a}))=a\wedge ac=a\wedge b$ or $f_c(\tau_c(\bar{a}))=ac\wedge acc=-c\wedge ac=-a\wedge b$ and we see that $f$ is a bijection on basis elements. Thus $f$ induces an isomorphism $W_1\to U$ and the Lemma follows.
\end{proof}

\begin{Le} Define the $G$-module homomorphisms $\theta_1:\Z[G]\to W_1$ and $\theta_2:\Z[G]\to W_2$ by 
$$\theta_1(g)=\sum_{o(x)=2}j_x(\bar{g}),$$
$$\theta_2(g)=\sum_\xi i_\xi(g-g\ss(\xi)\1).$$
Here $\bar{g}$ is the image of $g\in G$ in $\Z[G]/\Z[G](x+1)$. 	 
Then one has the following commutative diagram
$$\xymatrix{\Z[G]\ar[r] ^{\delta^1_{\mathcal N}} & \La^2(\Z[G])\\
	\Z[G]\ar[r]^{\theta} \ar[u]^{id}&  W_1 \oplus W_2\ar[u]^{f}
}$$
where  the components of $\theta$ are $\theta_1$ and $ \theta_2$.
\end{Le}
\begin{proof} Since the morphisms involved are $G$ homomorphisms, we only need to control the image of $1\in \Z[G]$. We have
$$\delta^1_{\mathcal N}(1)=-{\mathcal N}\wedge 1=\sum_{x\not =1} 1\wedge x= \sum_{o(x) =2} 1\wedge x+\sum_{x^2\not =1} 1\wedge x.$$
The second summand can be rewritten as follows. If $x^2\not=1$, then $x$ and $x\1$ are two representatives of a class in ${\mathfrak O}_2(G)$. Hence 
$$\delta^1_{\mathcal N}(1)=\sum_{o(x) =2} 1\wedge x+\sum_\xi \left (1\wedge \ss(\xi) +1\wedge \ss(\xi)\1\right).$$
On the other hand, $f\circ\theta=f_1\theta_1+f_2\theta_2$.
It follows that

\begin{align*}
f(\theta(1))& =f_1\left (\sum_{o(x)=2} j_x(\bar{1})\right )+ f_2\left ( \sum_\xi f_{\ss(\xi)} (1-\ss(\xi)\1)\right )\\ &=  \sum_{o(x)=2} \hat{f}_x(1)+\sum_\xi f_{\ss(\xi)}(1-\ss(\xi)\1)\\ &=\sum_{o(x)=2}1\wedge x+\sum_\xi \left( 1\wedge \ss(\xi)-\ss(\xi)\1\wedge \ss(\xi)\1\ss(\xi)\right)\\
&=\sum_{o(x)=2}1\wedge x+\sum_\xi \left( 1\wedge \ss(\xi)+1\wedge \ss(\xi)\1\right)
\end{align*}
and the result is proved.
\end{proof}

According to Lemma \ref{ttvkap}, we know that  $\varkappa_1:\hat{H}^{-1}(G,M)\to  H_1^\varkappa(G,M)$ is an isomorphism if $G$ has no elements of order two. Now we will consider arbitrary groups. To state the corresponding fact, we fix some notation. If $o(x)=q$ we denote by $C_q(x)$ the cyclic subgroup generated by $x$.  
\begin{Pro}\label{ox=2} Let $G$ be a finite group of order $n$ and $M$ be a $G$-module. Then one has the following exact sequence
	$$\bigoplus_{o(x)=2} \hat{H}^{-1}(C_2(x),M)\to \hat{H}^{-1}(G,M)\xto{\varkappa_1} H_1^\varkappa (G,M) \to 0.$$ 
\end{Pro}

\begin{proof} It follows from the above lemmata that the complex $K^\varkappa_*(G,M)$ in dimensions $0,1,2$ looks as follows:
$$\xymatrix{H^0(G,M)&M\ar[l] & W^1(G,M)\oplus W^2(G,M)\ar[l]
},  $$ 
where $$W^1(G,M)=\hom_G(W_1,M)$$
is the set of functions $\ell _1$  defined on the set $x\in G$ of elements of order $2$ with values in $M$ satisfying the condition $(x+1)\ell_1(x)=0$, while 
$$W^2(G,M)=\hom_G(W_2,M)$$
is the set of functions $\ell _1$  defined on the set $\xi\in {\mathfrak O}_2(G)$ with values in $M$.
The $0$-th boundary map $d:M\to H^0(G,M)$ is the norm map $m\mapsto \sum_{x\in G} xm$, while the next boundary map $d$ is given by
$$d(\ell_1,\ell_2)=\sum_{o(x)=2} \ell_1(x)+ \sum_{\xi} (1-\ss(\xi)\1)\ell_2(\xi).$$
After these preliminaries, we will now prove Proposition \ref{ox=2}. By definition we have
$$\hat{H}^{-1}(G,M) = \frac{Ker({\sf N}:M\to H^0(G,M))}{V} \quad {\rm and} \quad H^\varkappa_{1}(G,M) = \frac{Ker({\sf N}:M\to H^0(G,M))}{U},$$
where $V$ is generated by elements of the form $gm-m$, $g\in G$, $m\in M$, while $U=U_1+U_2$. Here $U_1$ is generated by $n\in M$, where $xn+n=0$ for an element $x\in G$ of order two and $U_2$ is generated by elements of the form $gm-m$,  where $g=\ss(\xi)$ for an element $\xi\in {\mathfrak O}_2(G)$ and $m\in M$. Our claim is that $V\subset U$. In other words we have to show $gm-m\in U$ for all $g\in G$ and $m\in M$. By definition, this holds automatically if $g=\ss(\xi)\1$ for an element $\xi \in {\mathfrak O}_2(G)$. Since $gm-m=-(g\1n-n)$ for $n=gm$, we see that  $gm-m\in U$ for all $g$ with $o(g)\not =2$. Assume now $o(g)=2$. Then for $n=gm-m$ we have
$(g+1)n=0$ and hence $n\in U_1\subset U$. So we proved that $V\subset U$. It implies that the map $\varkappa_1$ is surjective. It is also clear that $Ker(\varkappa_1)=U/V$. Since $U=U_1+ U_2$ and $U_2\subset V$, we see that $Ker(\varkappa_1)=U_1/U_1\cap V$. On the other hand the image of $\bigoplus_{o(x)=2} \hat{H}^{-1}(C_2(x),M)$ in $\hat{H}^{-1}(G,M)$ is $U_1/V_1$, where
$V_1\subset M$ is a subgroup generated by elements of the form $xn-n$, where $o(x)=2$, $n\in M$. Obviously $V_1\subset U_1\cap V$ and Proposition \ref{ox=2} follows. 
\end{proof}

\begin{Co}\label{601} Let $G$ be a finite group of order $n$ and $M$ be a $G$-module. Then one has the following exact sequence
$$\bigoplus_{C_2\subset G} \hat{H}^{-1}(C_2,M\tw)\to \hat{H}^{-1}(G,M\tw)\to H^{n-2}_\la (G,M) \to 0,$$
where the sum is taken over all subgroups of order two.
\end{Co}
This follows directly from  the Proposition \ref{ox=2} and the Poincar\'e duality (\ref{pd}). 


\end{document}